\documentclass{article}
\usepackage[mathscr]{euscript}
\usepackage{amsfonts}
\usepackage{amsmath}
\usepackage{amsthm}
\usepackage{amssymb}
\usepackage{hyperref}

\newtheorem{theorem}{Theorem}
\newtheorem{lemma}{Lemma}
\newtheorem{remark}{Remark}

\DeclareMathOperator{\sgn}{sgn}
\title{Marcinkiewicz Sampling Theorem for Orlicz Spaces}
\author{Aleksander Pawlewicz and Micha\l{} Wojciechowski}
\date{}

\makeatletter
\def\blfootnote{\gdef\@thefnmark{}\@footnotetext}
\makeatother

\begin{document}

\maketitle
\begin{abstract}
In the article we generalize the Marcinkiewicz sampling theorem in the context of Orlicz spaces. We establish conditions under which sampling theorem holds in terms of restricted submultiplicativity and supermultiplicativity of an $N$-function $\varphi$, boundedness of the Hilbert transform and Matuszewska-Orlicz indices. In addition we give a new criterion for boundedness of Hilbert transform on Orlicz space.

\end{abstract}
\blfootnote{\textup{2020} \textit{Mathematics Subject Classification}: 46E30, 47B65, 42A05.}

\section{Introduction}
In this paper we deal with complex-valued functions defined on the one dimensional torus group $\mathbb{T}=\mathbb{R}/2\pi\mathbb{Z}$. Throughout the paper we will use the following notation
$$x_{n,k}=\frac{2\pi(n+k)}{2n+1},$$
for $k= -n, -(n-1), ..., n-1, n$.

In the paper we present a generalisation of the Marcinkiewicz sampling theorem in the context of Orlicz spaces. In its original statement, this theorem says that for every trigonometric polynomial of degree $n\in\mathbb{N}=\{1, 2, 3, ...\}$, that is the function
$$f_n(x)=\sum_{k=-n}^n a_k e^{ikx},$$
where $a_{-n}, a_{-(n-1)}, ..., a_{n-1}, a_n$ are complex numbers, $f_n:\mathbb{T}\rightarrow\mathbb{C}$, we have
\begin{eqnarray}
\left(\frac{1}{2n+1}\sum_{k=-n}^n\left|f_n\left(x_{n,k}\right)\right|^p\right)^{1/p}
\leq
3\cdot\left(\frac{1}{2\pi}\int_\mathbb{T}|f_n(x)|^p\,dx\right)^{1/p}
\end{eqnarray}
for every $1\leq p\leq +\infty$, and
\begin{eqnarray}
\left(\frac{1}{2\pi}\int_\mathbb{T}|f_n(x)|^p\,dx\right)^{1/p}
\leq
C_p\left(\frac{1}{2n+1}\sum_{k=-n}^n\left|f_n\left(x_{n,k}\right)\right|^p\right)^{1/p}
\end{eqnarray}
for every $1<p<+\infty$, where $C_p>0$ is a constant depending only on $p$. For $p=1$, as the example of Dirichlet kernel shows, the inequality $(2)$ is not true. The above statement with proofs can be found in \cite[Theorem 9 and 10, pages 12 and 13]{Mar}, \cite[Theorem 1 and 2, page 132]{MaZy} and \cite[Volume II, pages 28 and 29]{Zyg}. For more information on this subject see \cite[Section 4.5.3, pages 222 - 224]{Ma2} and the references given there.

\smallskip

We begin with basic definitions of the objects used in the paper. An $N$-function $\varphi$ is a convex function from $[0,\infty)$ to $[0,\infty)$ such that
$$\varphi(t)=\int_0^t \phi(x)\,dx,$$
where $\phi:[0,\infty)\rightarrow[0,\infty)$ is a non-decreasing, right-continuous, strictly positive on $(0,\infty)$ function, such that $\phi(0)=0$ and $\lim_{t\rightarrow+\infty}\phi(t)=+\infty$. Then we have
\begin{equation}\label{zaleznosc}
\varphi(t)=\int_0^t\phi(s)\,ds\leq t\phi(t)\leq \int_t^{2t}\phi(s)\,ds\leq \varphi(2t).
\end{equation}
We will sometimes assume that the $N$-function $\varphi$ satisfies the $\Delta_2$ condition, that is there exists a constant $D>0$ such that for every $x>0$ we have
$$\varphi(2x)\leq D\varphi(x).$$
More information on this condition can be found for example in the book \cite[Chapter I, \S4, page 23]{KrRu}.

Given an $N$-function $\varphi$ and a positive measure $\mu$ on $\mathbb{T}$, we define the Orlicz space $L^\varphi(\mu)$ as a Banach space of $\mu$-measurable functions $f$ such that the Luxemburg norm of the function $f$,
$$\left\Vert f\right\Vert_{L^\varphi(\mu)}=\inf\left\{\lambda>0: \int_\mathbb{T}\varphi\left(\frac{|f(x)|}{\lambda}\right)\,d\mu\leq 1\right\},$$
is finite. In this article we will deal with three particular cases of the measure $\mu$, namely the Haar measure on $\mathbb{T}$ and two families of discrete measures. The Orlicz spaces considered in the paper are: $L^\varphi=\left\{f \in L^1, \left\Vert f\right\Vert_{L^\varphi}<\infty\right\},$ and $\ell_n^\varphi$, $L^\varphi(\omega_n)$ -
the spaces of trigonometric polynomials of degree $n$. 

Norms on these spaces are
$$\left\Vert f\right\Vert_{L^\varphi}=\inf\left\{\lambda>0: \frac{1}{2\pi}\int_\mathbb{T}\varphi\left(\frac{|f(x)|}{\lambda}\right)\,dx\leq 1\right\},$$
$$\left\Vert f\right\Vert_{\ell_n^\varphi}=\inf\left\{\lambda>0: \sum_{k=-n}^n\varphi\left(\frac{\left|f\left(x_{n,k}\right)\right|}{\lambda}\right)\leq 1\right\},$$
and
$$
\left\Vert f\right\Vert_{L^\varphi(\omega_n)}
=
\inf\left\{\lambda>0: \frac{1}{2n+1}\sum_{k=-n}^n\varphi\left(\frac{\left|f\left(x_{n,k}\right)\right|}{\lambda}\right)\leq 1\right\}.
$$
Here the positive, discrete measure $\omega_n$ is defined as follows
$$\omega_n=\frac{1}{2n+1}\sum_{k=-n}^n\delta_{x_{n,k}},$$
where $\delta_x$ is the Dirac measure at the point $x\in\mathbb{T}$.

Using the above notation we can formulate the main theorem.
\begin{theorem}\label{glowne}
\begin{enumerate}
\item
Let the $N$-function $\varphi$ satisfy the normalizing condition $\varphi(1)=1$ and
$$\varphi(a)\varphi(b)\leq\varphi(Cab)$$ 
for some $C>1$ and every $a<1\leq ab<b$ ($\varphi$ is restricted supermultiplicative).
Then for every trigonometric polynomial $f_n$ of degree $n$ we have
\begin{equation}
\left\Vert f_n\right\Vert_{\ell_n^\varphi}
\leq
6C^2\varphi^{-1}(2n+1)\left\Vert f_n\right\Vert_{L^\varphi}.
\end{equation}
\item
Let the $N$-function $\varphi$ satisfy the $\Delta_2$ condition, normalizing condition $\varphi(1)=1$,
$$\varphi(Cab)\leq\varphi(a)\varphi(b)$$ 
for some $1>C>0$ and every $a<1\leq ab<b$ ($\varphi$ is restricted submultiplicative),
and assume that the Hilbert transform is a bounded operator on the space $L^\varphi$.
Then there exists a constant $C_\varphi>0$ such that for every trigonometric polynomial $f_n$ of degree $n$ we have
\begin{equation}
C C_\varphi\varphi^{-1}(2n+1)\left\Vert f_n\right\Vert_{L^\varphi}
\leq
2\left\Vert f_n\right\Vert_{\ell_n^\varphi}.
\end{equation}
\end{enumerate}
\end{theorem}

\begin{remark}
Note that an $N$-function which is both restricted submultiplicative and restricted supmultiplicative is multiplicative, i.e. $\varphi(a)\varphi(b) \simeq \varphi(Cab)$. By \cite{Ma4} the only multiplicative functions are, up to equivalence, power functions.
\end{remark}

The paper is organized as follows. In Section 2 we formulate the auxiliary theorems used to prove Theorem \ref{glowne}. Section 3 contains the proof of auxiliary theorems. Section 4 is devoted to comparison of $\ell_n^\varphi$ and $L^\varphi(\omega_n)$ and the proof of Theorem \ref{glowne}.

\section{Auxiliary theorems}

In this section we will show the connections of the Marcinkiewicz sampling theorem for Orlicz spaces with probability measures with the Hilbert transform. For more information about this transform see \cite[Chapter 12, \S12.8, pages 90-99]{Ed}. Here we only state the definition. Let a distribution $H$ be given by a Fourier series 
\begin{equation}\label{dystrybucja}
H\backsim\sum_{n\in\mathbb{Z}}-i\sgn(n)e^{inx},
\end{equation}
where $\sgn(0)=0$ and $\sgn(n)=n/|n|$ for $n\neq 0$.
Then, the Hilbert transform of a function $f:\mathbb{T}\rightarrow\mathbb{C}$ is defined as the convolution of $f$ with the distribution $H$ in the principal value sense
$$H(f)(x)=H*f(x)=\lim_{\varepsilon\rightarrow 0}\frac{1}{2\pi}\int_{\varepsilon\leq|t|\leq\pi}f(x-t)\cot(t/2)\,dt.$$

Notice that, the form of Fourier series \eqref{dystrybucja} of the distribution $H$ means that the projection of function $f$ into the space of trigonometric polynomials of degree $n\in\mathbb{N}$ is given by the formula
$$
D_n(f)(x)=
e^{i(n+1)x}\left(\frac{I-iH}{2}\right)*\left(
e^{-2i(n+1)x}\left(\frac{I+iH}{2}\right)*\left(e^{i(n+1)x}f(x)\right)
\right).
$$
Here $I$ is the identity operator. The convolution operators $\frac{I-iH}{2}$ and $\frac{I+iH}{2}$ are the projections onto the space of functions which have only negative or positive Fourier coefficients respectively. Therefore, the norm of the operator $D_n$ satisfies:
\begin{equation}\label{rzut_norma}
||D_n||\leq \left(\frac{||I||+||H||}{2}\right)^2,
\end{equation}
where $||H||$ is the norm of the Hilbert transform.

In order to prove Theorem \ref{glowne}, first we will show the relations between norms on the spaces $L^\varphi$ and $L^\varphi(\omega_n)$. The next two theorems deal with this.

\begin{theorem}\label{Simple}
For every $n\in\mathbb{N}$, every N-function $\varphi$ and every trigonometric polynomial $f_n$ of degree $n$ we have
$$\left\Vert f_n\right\Vert_{L^\varphi(\omega_n)}\leq 3\left\Vert f_n\right\Vert_{L^\varphi}.$$
\end{theorem}

\begin{proof}
The proof follows immediately from the inequality
\begin{equation}\label{Zygmund}
\frac{1}{2n+1}\sum_{k=-n}^n\varphi\left(\frac{|f_n(x_{n,k})|}{3}\right)
\leq
\frac{1}{2\pi}\int_\mathbb{T}\varphi\left(|f_n(x)|\right)\,dx
\end{equation}
proved in the book \cite[Volume II, page 29]{Zyg}. 
\end{proof}

The next theorem imposes conditions on the function $\varphi$ that imply the reverse inequality.

\begin{theorem}\label{Hilbert_transform}
For every $N$-function $\varphi$ which satisfies the $\Delta_2$ condition, the following are equivalent:
\begin{enumerate}
\item there exists a constant $C_\varphi>0$ such that for every $n\in\mathbb{N}$ and every trigonometric polynomial $f_n$ of degree $n$ we have
$$\left\Vert f_n\right\Vert_{L^\varphi}\leq C_\varphi\left\Vert f_n\right\Vert_{L^\varphi(\omega_n)},$$
\item the Hilbert transform is a bounded operator on the space $L^\varphi$.
\end{enumerate}
\end{theorem}

It is well known that boundedness of the Hilbert transform can be expressed in terms of non-triviality of Boyd indices (which are equal to Matuszewska-Orlicz indices in case of Orlicz spaces - \cite[chapter 11]{Ma1}). The above theorem can be considered as one more equivalent condition for boundedness of the Hilbert transform.

\section{Proof of Theorem \ref{Hilbert_transform}}
We will need the following estimate for the Dirichlet kernel.
\begin{lemma}\label{Dirichlet}
Let 
$$D_n(x)=\sum_{k=-n}^ne^{ikx}=\frac{\sin\left({\frac{2n+1}{2}x}\right)}{\sin{\frac{x}{2}}}$$
be the $n$-th Dirichlet kernel for $n=0, 1, ...\,$. Moreover, let $\varphi$ be an $N$-function. Then we have
$$\int_0^\frac{3\cdot(2n+1)}{2\pi\lambda}\frac{\varphi(t)}{\lambda t^2}\,dt\leq \int_{\mathbb{T}}\varphi\left(\frac{|D_n(x)|}{\lambda}\right)\,dx \leq 4\pi\int_0^{\frac{2n+1}{\lambda}}\frac{\varphi(2t)}{\lambda t^2}\,dt,$$
for every $\lambda>0$.
\end{lemma}

\begin{proof}
We have
\begin{equation}\label{part_integration}
\int_{\mathbb{T}}\varphi\left(\frac{|D_n(x)|}{\lambda}\right)\,dx
=
\int_0^\frac{2n+1}{\lambda}\phi(t)|\{x\in\mathbb{T}:|D_n(x)|>\lambda t\}|\,dt
\end{equation}
and, by \eqref{zaleznosc}, 
\begin{equation}\label{part_integration_2}
\frac{\varphi(t)}{t}\leq \phi(t) \leq \frac{\varphi(2t)}{t},
\end{equation}
so we only need to estimate the value $|\{x\in\mathbb{T}:|D_n(x)|>t\}|$. We immediately get
\begin{equation}\label{upper}
\begin{aligned}
|\{x\in\mathbb{T}:|D_n(x)|>\lambda t\}| 
&=
\left|\left\{x\in\mathbb{T}: \left|\frac{\sin\left({\frac{2n+1}{2}x}\right)}{\sin{\frac{x}{2}}}\right|>\lambda t\right\}\right| \\
&=
2\left|\left\{x\in[0,\pi): \left|\frac{\sin\left({\frac{2n+1}{2}x}\right)}{\sin{\frac{x}{2}}}\right|>\lambda t\right\}\right| \\
&\leq
2\left|\left\{x\in[0,\pi): \left|\frac{1}{\frac{x}{\pi}}\right|>\lambda t\right\}\right| \\
&=
\frac{2\pi}{\lambda t}.
\end{aligned}
\end{equation}
For the lower estimate we calculate in this way
\begin{equation}\label{lower}
\begin{aligned}
|\{x\in\mathbb{T}:|D_n(x)|>\lambda t\}| 
&=
2\left|\left\{x\in[0,\pi): \left|\sin\left({\frac{2n+1}{2}x}\right)\right|>\lambda t\sin{\frac{x}{2}}\right\}\right| \\
&\geq
2\left|\left\{x\in[0,\pi): \left|\sin\left({\frac{2n+1}{2}x}\right)\right|\geq \frac{1}{2}>\lambda t\frac{x}{2}\right\}\right| \\
&\geq
\frac{1}{\lambda t},
\end{aligned}
\end{equation}
for $t<\frac{3\cdot(2n+1)}{2\pi\lambda}$. By putting estimates \eqref{part_integration}, \eqref{part_integration_2}, \eqref{upper} and \eqref{lower} together, we get
$$
\int_0^\frac{3\cdot(2n+1)}{2\pi\lambda} \frac{\varphi(t)}{\lambda t^2}\,dt
\leq
\int_{\mathbb{T}}\varphi\left(\frac{|D_n(x)|}{\lambda}\right)\,dx
\leq
2\pi\int_0^\frac{2n+1}{\lambda} \frac{\varphi(2t)}{\lambda t^2}\,dt.
$$ 
\end{proof}

\begin{remark}
Notice that Lemma \ref{Dirichlet} implies also that 
\begin{equation}\label{lambda}
||D_n||_{L^\varphi}\leq 4\pi ||D_{n+1}||_{L^\varphi}
\end{equation}
for $n=0, 1, ...\,$.
\end{remark}

The next lemma is in fact the most important in our proof.

\begin{lemma}\label{estimate}
Let an N-function $\varphi$ satisfy the $\Delta_2$ condition. Assume also that there exists a constant $C_\varphi>0$ such that for every $n\in\mathbb{N}$ and every trigonometric polynomial $f_n$ of degree $n$ we have
\begin{equation}\label{Marcinkiewicz}
\left\Vert f_n\right\Vert_{L^\varphi}\leq C_\varphi\left\Vert f_n\right\Vert_{L^\varphi(\omega_n)}.
\end{equation}
Then there exist constants $\sigma>0$, $\gamma>0$ and $p>1$ such that for every $s>0$ we have
\begin{equation}\label{small}
\frac{\sigma s}{\varphi(\sigma s)}\int_0^s\frac{\varphi(r)}{r^2}\,dr\leq \sigma,
\end{equation}
and 
\begin{equation}\label{big}
\frac{(\gamma s)^p}{\varphi(\gamma s)}\int_s^\infty\frac{\varphi(r)}{r^{p+1}}\,dr\leq \gamma^p.
\end{equation}
\end{lemma}

\begin{remark}
Conditions \eqref{small} and \eqref{big} in the context of Marcinkiewicz interpolation theorem were known before and can be found, for example, in the papers \cite[Chapter 12, Theorem 4.22, page 116]{Zyg}, \cite[Page 161]{Ma3}, \cite[Chapter 11, page 89]{Ma1}, \cite[Theorem, page 361; Example 3, page 363]{Cia}.
\end{remark}

\begin{remark}
As the proof of Theorem \ref{Hilbert_transform} shows, the conditions \eqref{small} and \eqref{big} imply that every sublinear operator which is of weak types $(1,1)$ and $(p,p)$ is bounded on the space $L^\varphi$.
\end{remark}

\begin{proof}[Proof of Lemma \ref{estimate}] $\,$\\
 Consider the $n$-th Dirichlet kernel $D_n$ (trigonometric polynomial of degree $2n+1$). Let $\lambda_n=||D_n||_{L^\varphi}$ for $n=0, 1, ...\,$. Then by \eqref{Marcinkiewicz} we get
\begin{equation}\label{3to1}
\begin{aligned}
\frac{1}{2\pi}\int_\mathbb{T}\varphi\left(\frac{|D_n(x)|}{\lambda_n}\right)\,dx
\leq
1
&\leq
\frac{1}{2n+1}\sum_{k=-n}^n\varphi\left(\frac{C_\varphi |D_n(x_{n,k})|}{\lambda_n}\right) \\
&=
\frac{\varphi\left(C_\varphi \frac{2n+1}{\lambda_n}\right)}{2n+1},
\end{aligned}
\end{equation}
because
$D_n(x_{n,-n})=2n+1 \mbox{ and } D_n(x_{n,k})=0 \mbox{ for } k=-(n-1), -(n-2), ..., n-1, n$.

By Lemma \ref{Dirichlet} we get
$$
\frac{1}{2\pi}\int_0^\frac{2n+1}{\lambda_n}\frac{\varphi(2t)}{\lambda_n t^2}\,dt
\leq
\frac{1}{2\pi}\int_{\mathbb{T}}\varphi\left(\frac{|D_n(x)|}{\lambda_n}\right)\,dx
\leq
\frac{\varphi\left(C_\varphi \frac{2n+1}{\lambda_n}\right)}{2n+1}.
$$
The above estimates can be rewritten in the following way
\begin{equation}\label{posrednie}
\begin{aligned}
\frac{C_\varphi\frac{2n+1}{\lambda_n}}{\varphi\left(C_\varphi\frac{2n+1}{\lambda_n}\right)}\int_0^\frac{2n+1}{\lambda_n}\frac{\varphi(2t)}{t^2}\,dt
\leq
2\pi C_\varphi.
\end{aligned}
\end{equation}

On the other hand, by \eqref{3to1} we have
$$
1
\leq
\frac{\varphi\left(C_\varphi \frac{2n+1}{\lambda_n}\right)}{2n+1},
$$
which implies that
$$
\frac{\varphi^{-1}(2n+1)}{C_\varphi}\leq\frac{2n+1}{\lambda_n}.
$$
This means that the value $\frac{2n+1}{\lambda_n}$ tends to infinity as $n\rightarrow\infty$.
Now let $s>0$ and let $n$ be the smallest natural number such that
$$
\frac{s}{2}\leq\frac{2n+1}{\lambda_n}.
$$
By the above definition 
$$
\frac{s}{2}\geq \frac{2n-1}{\lambda_{n-1}}.
$$
Then by \eqref{lambda} and the inequality above,
\begin{equation}\label{tylda}
\frac{2n+1}{\lambda_n}\leq4\pi \cdot\frac{2n+1}{\lambda_{n-1}}=4\pi \cdot\frac{2n-1}{\lambda_{n-1}}\cdot\frac{2n+1}{2n-1} < 2\pi\cdot\frac{2n+1}{2n-1}\cdot s \leq 21s,
\end{equation}
for every $n=1, 2, ...\,$.
Then, by $\Delta_2$ condition, 
$$\frac{\varphi(C_\varphi\frac{2n+1}{\lambda_n})}{\varphi(C_\varphi s)}\leq C,$$
and we have
\begin{equation}
\label{pomocnicze}
\begin{aligned}
2\frac{C_\varphi s}{\varphi(C_\varphi s)}\int_0^s\frac{\varphi(t)}{t^2}\,dt
&=
\frac{C_\varphi s}{\varphi(C_\varphi s)}\int_0^{s/2}\frac{\varphi(2t)}{t^2}\,dt \\
&\leq
\frac{\varphi(C_\varphi\frac{2n+1}{\lambda_n})}{\varphi(C_\varphi s)}
\frac{2C_\varphi\frac{2n+1}{\lambda_n}}{\varphi(C_\varphi\frac{2n+1}{\lambda_n})}\int_0^\frac{2n+1}{\lambda_n}\frac{\varphi(2t)}{t^2}\,dt \\
&\leq
C\frac{2C_\varphi\frac{2n+1}{\lambda_n}}{\varphi(C_\varphi\frac{2n+1}{\lambda_n})}\int_0^\frac{2n+1}{\lambda_n}\frac{\varphi(2t)}{t^2}\,dt \\
&\leq
C 4\pi C_\varphi,
\end{aligned}
\end{equation}
by the estimate \eqref{posrednie}.
Now setting $\sigma=2\pi C C_\varphi$
and using the fact, that $\varphi(\alpha t)\geq\alpha\varphi(t)$ for every convex, increasing function $\varphi$ such that $\varphi(0)=0$ and $\alpha\geq 1$ , $t\geq 0$,
we have
\begin{eqnarray*}
\frac{\sigma s}{\varphi(\sigma s)}\int_0^s\frac{\varphi(t)}{t^2}\,dt
&=&
\frac{2\pi C C_\varphi s}{\varphi(2\pi C C_\varphi s)}\int_0^s\frac{\varphi(t)}{t^2}\,dt \\
&\leq&
\frac{2\pi C C_\varphi s}{2\pi C\varphi(C_\varphi s)}\int_0^s\frac{\varphi(t)}{t^2}\,dt \\
&=&
\frac{C_\varphi s}{\varphi(C_\varphi s)}\int_0^s\frac{\varphi(t)}{t^2}\,dt \\
&\leq&
2\pi C C_\varphi = \sigma.
\end{eqnarray*}
In the last step we used estimate \eqref{pomocnicze}. So the inequality \eqref{small} follows. 

Now let us turn to the inequality \eqref{big}.  
Let $p>1$ and $s>0$. Then we have

\begin{equation}\label{duzy}
\begin{aligned}
\int_s^\infty \frac{\varphi(r)}{r^{p+1}}\,dr
&\leq&
\int_s^\infty\frac{1}{r^p} \int_0^r\frac{\phi(x)}{x}\,dx\,dr \\
&\leq&
\int_s^\infty \frac{1}{r^p} \int_0^r\frac{\varphi(2x)}{x^2}\,dx\,dr \\
&\leq&
D\int_s^\infty \frac{1}{r^p} \int_0^r\frac{\varphi(x)}{x^2}\,dx\,dr,
\end{aligned}
\end{equation}
for some constant $D>0$, by inequality \eqref{zaleznosc} and the $\Delta_2$ condition.
Now we have
\begin{eqnarray*}
\int_s^\infty \frac{\varphi(r)}{r^{p+1}}\,dr
&\leq&
D\int_s^\infty \frac{1}{r^p} \int_0^r\frac{\varphi(x)}{x^2}\,dx\,dr \\
&=&
D\int_s^\infty \frac{1}{r^p} 
\left[\int_0^s\frac{\varphi(x)}{x^2}\,dx +\int_s^r\frac{\varphi(x)}{x^2}\,dx\right]
\,dr \\
&\leq&
\frac{D}{p-1}\left[\frac{\varphi(\sigma s)}{s^p}+\int_s^\infty\frac{\varphi(x)}{x^{p+1}}\,dx\right].
\end{eqnarray*}
We applied inequality \eqref{small} for the first integral and Fubini's theorem for the second. The last inequality can be rewritten in the form
$$\left(1-\frac{D}{p-1}\right)\int_s^\infty\frac{\varphi(r)}{r^{p+1}}\,dr
\leq
\frac{D}{p-1}\cdot\frac{\varphi(\sigma s)}{s^p}.$$

Now, choosing $p=2D$ we get
$$\int_s^\infty\frac{\varphi(r)}{r^{p+1}}\,dr
\leq
\frac{D}{D-1}\cdot\frac{\varphi(\sigma s)}{s^p}
\leq
\frac{\varphi\left(\frac{D\sigma s}{D-1}\right)}{s^p},$$
(we used the inequality $\varphi(\alpha t)\geq\alpha\varphi(t)$ for every convex, increasing function $\varphi$ such that $\varphi(0)=0$ and $\alpha\geq 1$ , $t\geq 0$). Then it is enough to take $\gamma=\frac{D\sigma}{D-1}$ to get inequality \eqref{big}.
\end{proof}

Before the proof of Theorem \ref{Hilbert_transform}, we would like to remind the reader that if the Hilbert transform is a bounded operator on the Orlicz space $L^\varphi$, then it is also a bounded operator on the dual space to $L^\varphi$, that is the Orlicz space $L^{\varphi^*}$. This follows from a duality argument (the proof of this fact is the same as for $L^p$ spaces, see for example \cite[Chapter II, \S 2.5, page 33]{Ste}). For more information about relations between mutually complementary (convex conjugate) $N$-functions $\varphi$ and $\varphi^*$, see for example \cite[Chapter I, \S 2, pages 11-14]{KrRu} or \cite[Chapter II, \S 2.1, pages 13-15]{RaRe}.

\begin{proof}[Proof of Theorem \ref{Hilbert_transform}]
\bigskip

$1.\implies 2.$

In the first part of the proof, we will work in the same way as the standard proof of the ordinary Marcinkiewicz interpolation theorem goes, but in the second part of the proof, we will use estimates \eqref{small} and \eqref{big}.

From Lemma \ref{estimate} we get the existence of constants $\sigma>0$, $\gamma>0$ and $p>1$ such that for every $s>0$ inequalities \eqref{small} and \eqref{big} are satisfied. Without loss of generality, we may assume that $\sigma\geq\gamma$. The Hilbert transform $H$ is of weak type $(1,1)$ and of weak type $(p,p)$, that is there exists a constant $C>0$ such that for every $t>0$ we have
$$\left|\{|Hf|>t\}\right|\leq C\frac{\left\Vert f\right\Vert_{L^1}}{t}$$
and
$$\left|\{|Hf|>t\}\right|\leq \left(C\frac{\left\Vert f\right\Vert_{L^p}}{t}\right)^p.$$

Now assume that 
$$\left\Vert f\right\Vert_{L^\varphi}\leq 1.$$ 
Then, by inequality \eqref{zaleznosc} we get
\begin{equation}\label{slaby_typ}
\begin{aligned}
\frac{1}{2\pi}\int_\mathbb{T}\varphi\left(\frac{|Hf(x)|}{\sigma}\right)\,dx
&=
\frac{1}{2\pi}\int_0^\infty\phi(t)\left|\left\{\frac{|Hf(x)|}{\sigma}>t\right\}\right|\,dt \\
&\leq
\frac{1}{2\pi}\int_0^\infty\frac{\varphi(2t)}{t}\left|\left\{|Hf(x)|>\sigma t\right\}\right|\,dt, \\
&=
\frac{1}{2\pi}\int_0^\infty\frac{\varphi(t)}{t}\left|\left\{|Hf(x)|>\sigma t/2\right\}\right|\,dt, \\
\end{aligned}
\end{equation}
Now define
 \begin{equation*}
    f^\nu(x) =
    \begin{cases}
      f(x) & \mbox{if } |f(x)|\geq\nu \\
      0        & \mbox{otherwise}
    \end{cases}
 \end{equation*}
and 
$$f_\nu(x)=f(x)-f^\nu(x).$$
Let us take $\nu=\sigma t$. Then we have
\begin{eqnarray*}
\left|\left\{|Hf(x)|>\sigma t/2\right\}\right|
&\leq&
\left|\left\{|Hf^{\sigma t}(x)|>\sigma t/4\right\}\right| + \left|\left\{|Hf_{\sigma t}(x)|>\sigma t/4\right\}\right| \\
&\leq&
4C\frac{\left\Vert f^{\sigma t}\right\Vert_{L^1}}{\sigma t} + 4^pC^p\frac{\left\Vert f_{\sigma t}\right\Vert_{L^p}^p}{\sigma^pt^p}.
\end{eqnarray*}

Combining the above estimates with inequality \eqref{slaby_typ}, we get
\begin{equation}\label{raz}
\frac{1}{2\pi}\int_\mathbb{T}\varphi\left(\frac{|Hf(x)|}{\sigma}\right)\,dx
\leq
\frac{1}{2\pi}\int_0^\infty\frac{\varphi(t)}{t}\left[4C\frac{\left\Vert f^{\sigma t}\right\Vert_{L^1}}{\sigma t} + 4^pC^p\frac{\left\Vert f_{\sigma t}\right\Vert_{L^p}^p}{\sigma^pt^p}\right]\,dt.
\end{equation}

By Fubini's theorem, inequality \eqref{small} and the fact that $\left\Vert f\right\Vert_{L^\varphi}\leq 1$, we have (here $\chi_A$ denotes the characteristic function of the set $A$)
\begin{equation}\label{dwa}
\begin{aligned}
\int_0^\infty\frac{\varphi(t)}{t^2} \frac{\left\Vert f^{\sigma t}\right\Vert_{L^1}}{\sigma}\,dt
&=
\frac{1}{2\pi}\int_0^\infty\frac{\varphi(t)}{t^2}\frac{1}{\sigma}\int_\mathbb{T}|f(x)|\chi_{\{y: |f(y)|\geq \sigma t\}}(x)\,dx\,dt, \\
&=
\frac{1}{2\pi}\int_\mathbb{T}\frac{|f(x)|}{\sigma}\int_0^\infty\frac{\varphi(t)}{t^2}\chi_{\{y: |f(y)|\geq \sigma t\}}(x)\,dt\,dx, \\
&=
\frac{1}{2\pi}\int_\mathbb{T}\frac{|f(x)|}{\sigma}\int_0^\frac{|f(x)|}{\sigma}\frac{\varphi(t)}{t^2}\,dt\,dx, \\
&\leq
\frac{1}{2\pi}\int_\mathbb{T}\frac{|f(x)|}{\sigma}\frac{\varphi\left(\sigma\frac{|f(x)|}{\sigma}\right)}{\frac{|f(x)|}{\sigma}}\,dx, \\
&=
\frac{1}{2\pi}\int_\mathbb{T} \varphi\left(|f(x)|\right) \,dx 
\leq
1.
\end{aligned}
\end{equation}

Similarly, by Fubini's theorem, inequality \eqref{big}, the assumption that $\gamma\leq\sigma$ and the inequality $\left\Vert f\right\Vert_{L^\varphi}\leq 1$ we have
\begin{equation}\label{trzy}
\begin{aligned}
\int_0^\infty\frac{\varphi(t)}{t^{p+1}} \frac{\left\Vert f_{\sigma t}\right\Vert_{L^p}^p}{\sigma^p}\,dt
&=
\frac{1}{2\pi}\int_0^\infty\frac{\varphi(t)}{t^{p+1}}\frac{1}{\sigma^p}\int_\mathbb{T}|f(x)|^p\chi_{\{y: |f(y)|< \sigma t\}}(x)\,dx\,dt, \\
&=
\frac{1}{2\pi}\int_\mathbb{T}\frac{|f(x)|^p}{\sigma^p}\int_0^\infty\frac{\varphi(t)}{t^{p+1}}\chi_{\{y: |f(y)|< \sigma t\}}(x)\,dt\,dx, \\
&=
\frac{1}{2\pi}\int_\mathbb{T}\frac{|f(x)|^p}{\sigma^p}\int_\frac{|f(x)|}{\sigma}^\infty\frac{\varphi(t)}{t^{p+1}}\,dt\,dx, \\
&\leq
\frac{1}{2\pi}\int_\mathbb{T}\frac{|f(x)|^p}{\sigma^p}\frac{\varphi\left(\gamma\frac{|f(x)|}{\sigma}\right)}{\frac{|f(x)|^p}{\sigma^p}}\,dt\,dx, \\
&\leq
\frac{1}{2\pi}\int_\mathbb{T} \varphi\left(|f(x)|\right) \,dx\leq
1.
\end{aligned}
\end{equation}

Putting estimates \eqref{raz}, \eqref{dwa} and \eqref{trzy} together, we get
$$
\frac{1}{2\pi}\int_\mathbb{T}\varphi\left(\frac{|Hf(x)|}{\frac{4C+(4C)^p}{2\pi}\sigma}\right)\,dx
\leq
\frac{1}{\frac{4C+(4C)^p}{2\pi}}\frac{1}{2\pi}\int_\mathbb{T}\varphi\left(\frac{|Hf(x)|}{\sigma}\right)\,dx 
\leq
1,
$$
by the convexity of the function $\varphi$ and the fact that $\frac{4C+(4C)^p}{2\pi}>1$. This means that the norm of the Hilbert transform as an operator from the Orlicz space $L^\varphi$ to the Orlicz space $L^\varphi$ is bounded by the constant
$$\frac{4C+(4C)^p}{2\pi}\sigma.$$\\

$2.\implies 1.$

In fact, we will repeat the proof from Zygmund's book for the $L^p$ spaces, $p\in(1,\infty)$, see \cite[page 29]{Zyg}. Let $f$ be a trigonometric polynomial of degree $n$ and $g$ be a function from the space $L^{\varphi^*}$. Let $D_n(g)$ denotes the projection of the function $g$ onto the space of trigonometric polynomial of degree $n$. By \eqref{rzut_norma}, it is a bounded operator on the space $L^{\varphi^*}$ because the Hilbert transform is a bounded operator on this space. Now we can write 
\begin{eqnarray*}
\left|\frac{1}{2\pi}\int_\mathbb{T} f(t) g(t)\, dt\right|
&=&
\left|\frac{1}{2\pi}\int_\mathbb{T} f(t) D_n(g)(t)\, dt\right| \\
&=&
\left|\widehat{\left(f D_n(g)\right)}(0)\right| \\
&=&
\left|\frac{1}{2n+1}\sum_{k=-n}^n (fD_n(g))\left(x_{n,k}\right)\right| \\
&\leq&
\left\Vert fD_n(g)\right\Vert_{L^1(\omega_n)} \\
&\leq&
2\left\Vert f\right\Vert_{L^\varphi(\omega_n)} 
\left\Vert D_n(g)\right\Vert_{L^{\varphi^*}(\omega_n)},
\end{eqnarray*}
by H\"older's inequality (see \cite[Chapter II, \S9.7 The Luxemburg norm, pages 78-80]{KrRu}), and further
\begin{eqnarray*}
\left|\frac{1}{2\pi}\int_\mathbb{T} f(t) g(t)\, dt\right|
&\leq&
2\left\Vert f\right\Vert_{L^\varphi(\omega_n)} 
\left\Vert D_n(g)\right\Vert_{L^{\varphi^*}(\omega_n)} \\
&\leq&
2\left\Vert f\right\Vert_{L^\varphi(\omega_n)} 
3\left\Vert D_n(g)\right\Vert_{L^{\varphi^*}} \\
&\lesssim&
\left\Vert f\right\Vert_{L^\varphi(\omega_n)} 
\left\Vert g\right\Vert_{L^{\varphi^*}},
\end{eqnarray*}
by Theorem \ref{Simple} and the boundedness of the projection operator on the space $L^{\varphi^*}$. Taking the supremum over functions $g$ such that $\left\Vert g\right\Vert_{L^{\varphi^*}}\leq 1$ we get
\begin{eqnarray*}
\left\Vert f\right\Vert_{L^\varphi}
&=&
\sup\left\{\int_{\mathbb{T}}f(t)g(t)\,dt: g\in L^{\varphi^*}, \left\Vert g\right\Vert_{L^\varphi}\leq 1\right\} \\
&\lesssim&
\left\Vert f\right\Vert_{L^\varphi(\omega_n)}, 
\end{eqnarray*}
and the constant depends only on the norm of the projection operator on the appropriate Orlicz space. 
\end{proof}

\begin{remark}
In notes \cite[Chapter 11, Theorem 11.8, pages 90 and 91]{Ma1} it is proven that the following conditions are equivalent:
\begin{itemize}
\item there exist constants $\sigma>0$, $\gamma>0$ and $p>1$ such that for every $s>0$ we have
$$
\frac{\sigma s}{\varphi(\sigma s)}\int_0^s\frac{\varphi(r)}{r^2}\,dr\leq \sigma,
$$
and 
$$
\frac{(\gamma s)^p}{\varphi(\gamma s)}\int_s^\infty\frac{\varphi(r)}{r^{p+1}}\,dr\leq \gamma^p.
$$
\item There exists a real number $p>1$ such that the Matuszewska-Orlicz indices
$$\alpha_\varphi=\lim_{t\rightarrow 0^+}\frac{\log\left(\sup_{s>0}\frac{\varphi(ts)}{\varphi(s)}\right)}{\log(t)}
=\lim_{t\rightarrow +\infty}\frac{\log\left(\inf_{s>0}\frac{\varphi(ts)}{\varphi(s)}\right)}{\log(t)}$$
and
$$\beta_\varphi=\lim_{t\rightarrow +\infty}\frac{\log\left(\sup_{s>0}\frac{\varphi(ts)}{\varphi(s)}\right)}{\log(t)}$$
satisfy inequalities
$$1<\alpha_\varphi\leq\beta_\varphi<p.$$
\end{itemize}
\end{remark}

\section{Comparison of $\ell_n^\varphi$ and $L^\varphi(\omega_n)$ and the proof of Theorem \ref{glowne}}
In this section we first prove an auxiliary theorem, and then we show that Theorem 1 follows directly from it and from previously proved theorems.
\begin{theorem}\label{porownanie}
\begin{enumerate}
\item
Let the $N$-function $\varphi$ satisfy the normalization condition $\varphi(1)=1$ and
$$\varphi(a)\varphi(b)\leq\varphi(Cab)$$ 
for some $C>1$ and every $a<1\leq ab<b$ ($\varphi$ is restricted supermultiplicative).
Then for every trigonometric polynomial $f_n$ of degree $n$ we have
\begin{equation}
\left\Vert f\right\Vert_{\ell_n^\varphi}
\leq
2C^2\varphi^{-1}(2n+1)\left\Vert f\right\Vert_{L^\varphi(\omega_n)}.
\end{equation}
\item
Let the $N$-function $\varphi$ satisfy the normalization condition $\varphi(1)=1$ and
$$\varphi(Cab)\leq\varphi(a)\varphi(b)$$ 
for some $1>C>0$ and $a<1\leq ab<b$.
Then for every trigonometric polynomial $f_n$ of degree $n$ we have
\begin{equation}
C\varphi^{-1}(2n+1)\left\Vert f\right\Vert_{L^\varphi(\omega_n)}
\leq
2\left\Vert f\right\Vert_{\ell_n^\varphi}.
\end{equation}
\end{enumerate}
\end{theorem}

\begin{proof}

Part 1).
We consider two cases. Let us define
$$B_n=\left\{k\in\mathbb{Z}: \frac{|f(x_{n,k})|}{||f||_{L^\varphi(\omega_n)}}\geq C, -n\leq k\leq n\right\},$$
and
$$S_n=\left(\mathbb{Z}\cap[-n,n]\right)\setminus B_n.$$

For the set $B_n$ we will use the assumption: $\varphi(a)\varphi(b)\leq\varphi(Cab)$ for $a<1\leq ab<b$. Substituting 
$$p=Cab, q=b,$$
we can rewrite restricted supermultiplicativity condition in the form 
\begin{equation}
\label{nier_mult}
\varphi\big(\frac{p}{Cq}\big)\leq\frac{\varphi(p)}{\varphi(q)},
\end{equation}
whenever
\begin{equation}
\label{nier_mult_warunki}
\frac{p}{Cq}<1\leq \frac{p}{C}<q.
\end{equation}
We set 
$$p=\frac{|f(x_{n,k})|}{||f||_{L^\varphi(\omega_n)}} \mbox{ and } q=\varphi^{-1}(2n+1).$$

First, we check that the above $p$ and $q$ satisfies \eqref{nier_mult_warunki}.
For $k=-n, -n+1, ..., n-1,\\ n$ we have by the definition of the $L^\varphi(\omega_n)$ norm
$$
\frac{1}{2n+1}\varphi\left(\frac{|f(x_{n,k})|}{||f||_{L^\varphi(\omega_n)}}\right)
\leq
\frac{1}{2n+1}\sum_{l=-n}^n\varphi\left(\frac{|f(x_{n,l})|}{||f||_{L^\varphi(\omega_n)}}\right)
\leq
1.
$$
This can be rewritten in the form:
$$\frac{|f(x_{n,k})|}{||f||_{L^\varphi(\omega_n)}}\leq\varphi^{-1}(2n+1),$$
which confirms \eqref{nier_mult_warunki}.

By \eqref{nier_mult} we get for $p=\frac{|f(x_{n,k})|}{||f||_{L^\varphi(\omega_n)}} \mbox{ and } q=\varphi^{-1}(2n+1)$ that
$$
\varphi\left(\frac{\left|f(x_{n,k})\right|}{C\varphi^{-1}(2n+1)||f||_{L^\varphi(\omega_n)}}\right)
\leq
\frac{\varphi\left(\frac{\left|f(x_{n,k})\right|}{||f||_{L^\varphi(\omega_n)}}\right)}{\varphi(\varphi^{-1}(2n+1))}.
$$
Summing the above inequalities over $k\in B_n$ we get
\begin{eqnarray*}
\sum_{k\in B_n}\varphi\left(\frac{\left|f(x_{n,k})\right|}{C\varphi^{-1}(2n+1)||f||_{L^\varphi(\omega_n)}}\right)
&\leq&
\frac{1}{2n+1}\sum_{k\in B_n}\varphi\left(\frac{\left|f(x_{n,k})\right|}{||f||_{L^\varphi(\omega_n)}}\right) \\
&\leq&
\frac{1}{2n+1}\sum_{k=-n}^n\varphi\left(\frac{\left|f(x_{n,k})\right|}{||f||_{L^\varphi(\omega_n)}}\right) \\
&\leq&
1,
\end{eqnarray*}
by the definition of the norm $||f||_{L^\varphi(\omega_n)}$. The above calculation means that
\begin{equation}\label{Bn}
\begin{aligned}
||f\chi_{B_n}||_{l_n^\varphi}
&=
\inf\Big\{\lambda>0: \sum_{k\in B_n} \varphi\left(\frac{|f(x_{n,k})|}{\lambda}\right)\leq 1\Big\} \\
&\leq
C\varphi^{-1}(2n+1)||f||_{L^\varphi(\omega_n)}.
\end{aligned}
\end{equation}

Now we consider the set $S_n$. By its definition we get
\begin{equation}\label{Sn}
\begin{aligned}
||f\chi_{S_n}||_{\ell_n^\varphi}
&=
\inf\Big\{\lambda>0: \sum_{k\in S_n} \varphi\left(\frac{|f(x_{n,k})|}{\lambda}\right)\leq 1\Big\} \\
&\leq
\inf\Big\{\lambda>0: \sum_{k\in S_n} \varphi\left(\frac{C||f||_{L^\varphi(\omega_n)}}{\lambda}\right)\leq 1\Big\} \\
&\leq
\inf\Big\{\lambda>0: (2n+1)\varphi\left(\frac{C||f||_{L^\varphi(\omega_n)}}{\lambda}\right)\leq 1\Big\} \\
&=
\inf\Big\{\lambda>0: \frac{C||f||_{L^\varphi(\omega_n)}}{\varphi^{-1}\left(\frac{1}{2n+1}\right)}\leq \lambda\Big\} \\
&=
C\frac{||f||_{L^\varphi(\omega_n)}}{\varphi^{-1}\left(\frac{1}{2n+1}\right)} \\
&\leq
C^2\varphi^{-1}(2n+1) ||f||_{L^\varphi(\omega_n)}.
\end{aligned}
\end{equation}
In the last inequality, we used the fact that $1\leq C\varphi^{-1}(1/x)\cdot\varphi^{-1}(x)$ for $x\geq 1$. 
It is a consequence of assumptions: $\varphi(a)\varphi(b)\leq\varphi(Cab)$, $a<1\leq ab<b$, and $\varphi(1)=1$ if we take
$$a=\varphi^{-1}\left(\frac{1}{x}\right) \mbox{ and } b=\varphi^{-1}(x),$$
for $x\geq 1$.

Combining \eqref{Bn} and \eqref{Sn} we have
\begin{eqnarray*}
||f||_{l_n^\varphi}
&\leq&
||f\chi_{B_n}||_{l_n^\varphi}+||f\chi_{S_n}||_{l_n^\varphi} \\
&\leq&
C\varphi^{-1}(2n+1)||f||_{L^\varphi(\omega_n)}+C^2\varphi^{-1}(2n+1)||f||_{L^\varphi(\omega_n)} \\
&=&
2C^2\varphi^{-1}(2n+1)||f||_{L^\varphi(\omega_n)}.
\end{eqnarray*}

Part 2).
Here, we also need to consider two cases. Let us define
$$B_n=\{k\in\mathbb{Z}: \varphi^{-1}(2n+1)\frac{|f(x_{n,k})|}{||f||_{\ell_n^\varphi}}\geq 1, -n\leq k\leq n\},$$
and
$$S_n=\left(\mathbb{Z}\cap[-n,n]\right)\setminus B_n.$$

For the set $B_n$ we will use the assumption: 
\begin{equation}
\label{war41}
\frac{\varphi(Cab)}{\varphi(b)}\leq\varphi(a)
\end{equation}
whenever
\begin{equation}
\label{war42}
a<1\leq ab<b.
\end{equation}
We set 
$$a=\frac{|f(x_{n,k})|}{||f||_{\ell_n^\varphi}} \mbox{ and } b=\varphi^{-1}(2n+1),$$
First, we check that the above $a$ and $b$ satisfy \eqref{war41}.
For $k=-n, -n+1, ..., n-1, n$ we have by the definition of the $l_n^\varphi$ norm
$$
\varphi\left(\frac{|f(x_{n,k})|}{||f||_{\ell_n^\varphi}}\right)
<
\sum_{l=-n}^n\varphi\left(\frac{|f(x_{n,l})|}{||f||_{\ell_n^\varphi}}\right)
\leq
1.
$$
This inequality can be rewritten in the form:
$$\frac{|f(x_{n,k})|}{||f||_{\ell_n^\varphi}}<1.$$
If there is only one non-zero $x_{n,k}$, $k=-n, ..., n$, then we can use directly the definitions of $\lVert\cdot\rVert_{L^\varphi(\omega_n)}$ and 
$\lVert\cdot\rVert_{\ell_n^\varphi}$ norms to get the desired estimate.

By \eqref{war41} we have for $a=\frac{|f(x_{n,k})|}{||f||_{\ell_n^\varphi}} \mbox{ and } b=\varphi^{-1}(2n+1)$ that 
$$
\frac{\varphi\left(C\varphi^{-1}(2n+1)\frac{\left|f(x_{n,k})\right|}{||f||_{\ell_n^\varphi}}\right)}{\varphi\left(\varphi^{-1}(2n+1)\right)}
\leq
\varphi\left(\frac{\left|f(x_{n,k})\right|}{||f||_{\ell_n^\varphi}}\right).$$
Summing over $k\in B_n$ the above inequalities we get
\begin{eqnarray*}
\frac{1}{2n+1}\sum_{k\in B_n}\varphi\left(C\varphi^{-1}(2n+1)\frac{\left|f(x_{n,k})\right|}{||f||_{\ell_n^\varphi}}\right)
&=&
\sum_{k\in B_n}\frac{\varphi\left(C\varphi^{-1}(2n+1)\frac{\left|f(x_{n,k})\right|}{||f||_{\ell_n^\varphi}}\right)}{\varphi\left(\varphi^{-1}(2n+1)\right)} \\
&\leq&
\sum_{k\in B_n}\varphi\left(\frac{\left|f(x_{n,k})\right|}{||f||_{\ell_n^\varphi}}\right) \\
&\leq&
\sum_{k=-n}^n\varphi\left(\frac{\left|f(x_{n,k})\right|}{||f||_{\ell_n^\varphi}}\right) \\
&\leq&
1,
\end{eqnarray*}
by the definition of the norm $||f||_{\ell_n^\varphi}$. We get 
\begin{equation}\label{Bn2}
\begin{aligned}
||f\chi_{B_n}||_{L^\varphi(\omega_n)}
&=
\inf\Big\{\lambda>0: \frac{1}{2n+1}\sum_{k\in B_n} \varphi\left(\frac{|f(x_{n,k})|}{\lambda}\right)\leq 1\Big\} \\
&\leq
\frac{||f||_{\ell_n^\varphi}}{C\varphi^{-1}(2n+1)}.
\end{aligned}
\end{equation}

Now look at the set $S_n$. We use its definition to get
\begin{equation}\label{Sn2}
\begin{aligned}
||f\chi_{S_n}||_{L^\varphi(\omega_n)}
&=
\inf\Big\{\lambda>0: \frac{1}{2n+1}\sum_{k\in S_n} \varphi\left(\frac{|f(x_{n,k})|}{\lambda}\right)\leq 1\Big\} \\
&\leq
\inf\Big\{\lambda>0: \frac{1}{2n+1}\sum_{k\in S_n} \varphi\left(\frac{||f||_{\ell_n^\varphi}}{\varphi^{-1}(2n+1)\lambda}\right)\leq 1\Big\} \\
&\leq
\inf\Big\{\lambda>0: \varphi\left(\frac{||f||_{\ell_n^\varphi}}{\varphi^{-1}(2n+1)\lambda}\right)\leq 1\Big\} \\
&\leq
\inf\Big\{\lambda>0: \frac{||f||_{\ell_n^\varphi}}{\varphi^{-1}(2n+1)\varphi^{-1}(1)}\leq \lambda\Big\} \\
&=
\frac{||f||_{\ell_n^\varphi}}{\varphi^{-1}(2n+1)},
\end{aligned}
\end{equation}
because $\varphi(1)=1$.

Combining \eqref{Bn2} and \eqref{Sn2} we have
\begin{eqnarray*}
||f||_{L^\varphi(\omega_n)}
&\leq&
||f\chi_{B_n}||_{L^\varphi(\omega_n)}+||f\chi_{S_n}||_{L^\varphi(\omega_n)} \\
&\leq&
\frac{1}{C\varphi^{-1}(2n+1)} ||f||_{\ell_n^\varphi}+\frac{1}{C\varphi^{-1}(2n+1)} ||f||_{\ell_n^\varphi} \\
&=&
\frac{2}{C\varphi^{-1}(2n+1)} ||f||_{\ell_n^\varphi}.
\end{eqnarray*}
The proof is complete.
\end{proof}

Now we are ready to prove Theorem \ref{glowne}.

\begin{proof}[Proof of Theorem \ref{glowne}.]
Point 1 follows from Theorem \ref{Simple} and point 1 of Theorem \ref{porownanie} while point 2 follows from Theorem \ref{Hilbert_transform} and point 2 of Theorem \ref{porownanie}.
\end{proof}

\section{Final remarks}
 
1) The sufficiency of the restricted supermultiplicativity condition is in fact not very surprising. For $N$-functions satisfying the $\Delta_2$ condition it is also necessary for the sampling theorem (Theorem \ref{porownanie}) to hold. Indeed, let
$f_{n,k}$ be a trigonometric polynomial of degree $2n+1$ taking value 1 at $k$ points from the set $\{x_{n,k}\}$ and 0 at remaining $2n-k+1$ points. By the definition of  Orlicz norms, we get

$$
k\cdot\varphi\left(\frac1{\|f_{n,k}\|_{\ell^\varphi_n}}\right)=1
$$
and
$$
\frac{k}{2n+1}\cdot\varphi\left(\frac1{\|f_{n,k}\|_{L^\phi(\omega_n)}}\right)=1.
$$
These are equivalent to
$$
\|f_{n,k}\|_{\ell^\varphi_n}=\frac{1}{\varphi^{-1}\left(\frac1k\right)}
$$
and
$$
\|f_{n,k}\|_{L^\varphi(\omega_n)}=\frac{1}{\varphi^{-1}\left(\frac{2n+1}k\right)}.
$$
If the sampling inequality is satisfied we get
$$
\frac{1}{\varphi^{-1}\left(\frac1k\right)}\leq C \varphi^{-1}(2n+1) \frac{1}{\varphi^{-1}\left(\frac{2n+1}k\right)},
$$
which is equivalent to
$$
\varphi^{-1}\left(\frac{2n+1}k\right)\leq C \varphi^{-1}\left(\frac1k\right)\varphi^{-1}(2n+1).
$$

Setting now $a=\varphi^{-1}\left(\frac1k\right)$ and $b=\varphi^{-1}(2n+1)$ we get 
$$
\varphi(a)\varphi(b)\leq \varphi(Cab)
$$
and using $\Delta_2$ condition we can release the restrictions on the special form of $a$ and $b$ paying the price of the bigger constant.

A little bit more complicated argument may be used to show also the necessity of supermultiplicativity condition in Theorem \ref{glowne}.

\vskip3mm
2) It is well known that any $N$-function defined for arguments greater than 1 and any other $N$-function defined for arguments less then 1 could be "glued" together to some function equivalent to $N$-function defined on all positive arguments (i.e. there exists an $N$-function which coincides with the first one on some neighbourhood of infinity and with the second one in the vicinity of 0). Hence, given the $N$-function  $\Psi$ for arguments bigger then 1, one can ask the question what is the greatest $N$-function $\psi$
defined for small arguments such that for the  glued function the sampling inequality (Theorem 4.1) holds. In view of Theorem \ref{glowne} one gets immediately
$$
\psi(t)\leq C\inf_{x>t^{-1}}\frac{\Psi(tx)}{\Psi(x)}.
$$
It follows that the biggest $N$-function $\psi$ satisfying the sampling inequality for $\Psi$ coincides with the biggest convex function smaller then the function defined on the right hand side of the above inequality. We will call it \emph{a sampling function} of $\Psi$.
Note that for $\Psi$ satisfying $\Delta_2$ condition its sampling function is well defined and unique up to equivalence.
Since equivalent $N$-functions define isomorphic Orlicz spaces, then in this way we can uniquely (up to the isomorphism) assign 
to any Orlicz space with $\Delta_2$ property its sampling sequential Orlicz space. 
For example, simple calculations shows that for $N$-functions $x\mapsto x^\alpha\log^\beta(1+x)$ its sampling function is $t\mapsto t^\alpha\log^{-\beta}(1+t^{-1})$ if $\beta>0$ and $t\mapsto t^\alpha$ for $\alpha>1$ and $\beta<0$. For $N$-function
$x\mapsto x^\alpha\log^\beta(1+x)\log^\gamma(1+\log(1+x))$ its sampling function is $t\mapsto t^\alpha\log^{-\beta}(1+t^{-1})\log^{-\gamma}(1+\log(1+t^{-1}))$ if $\beta$ and $t\mapsto t^\alpha$ if $\alpha>1$ and
$\beta<0$. More generally, if $\eta(x)$ is a
concave function growing to infinity slower than any power function such that the function $y\mapsto y(\log \eta(y))'$ is decreasing, then the sampling function of the $N$-function $x\mapsto x^\alpha\cdot\eta(x)$ is $t\mapsto t^\alpha \eta(t^{-1})^{-1}$.
Surprisingly, as one can see from the above examples the sampling space for $N$-function logaritmically smaller than the power function is the same as for the power function itself. This means that in this case the original Marcinkiewicz theorem is just stronger.
Note that for an $N$-function without $\Delta_2$ property its sampling space obtained in this manner is always trivial.
\vskip3mm
3) One can proceed using similar considerations as above for the interpolation inequality (Theorem 4.2). In this case the $N$-function for $t<1$ has to satisfy
$$
\psi(t)\geq C\sup_{x>t^{-1}}\frac{\Psi(tx)}{\Psi(x)}
$$
While one can easily calculate those "interpolating functions" for the specific examples as above, it is however not clear if there always exists the unique up to equivalence the smallest convex majorant of the right hand side of the above inequality.
\section*{Acknowledgements}
The authors would like to express their gratitude to Karol Le\'snik for interesting and helpful discussions.


\begin{thebibliography}{9}

		\bibitem{Cia} 
		Andrea Cianchi, 
		\textit{An Optimal Interpolation Theorem of Marcinkiewicz Type in Orlicz Spaces},
		Journal of Functional Analysis 153, 357-381, 
		1998.

		\bibitem{Ed}
		Robert E. Edwards
		\textit{Fourier Series, a modern introduction},
		Second Edition, Volume 2;
		Graduate Texts in Mathematics, Vol. 85;
		Springer-Verlag,
		1982.

		\bibitem{KrRu}
		Mark A. Krasnosel'skii, Yakov B. Rutickii
		\textit{Convex Functions and Orlicz Spaces},
		P. Noordfoff Ltd., Groningen, The Netherlands
		1961.

		\bibitem{Ma1}
		Lech Maligranda
		\textit{Orlicz Spaces and Interpolation},
		Seminars in mathematics,
		Departamento de Matem\'atica, Universidade Estadual de Campinas,
		Campinas
		1989.

		\bibitem{Ma2}
		Lech Maligranda
		\textit{Józef Marcinkiewicz (1910-1940) - on the century of his birth},
		Banach Center Publications 95 (1), 133-234
		Institute of Mathematics Polish Academy of Sciences,
		2011.

		\bibitem{Ma3}
		Lech Maligranda
		\textit{Marcinkiewicz Interpolation Theorem and Marcinkiewicz Spaces},
		Wiadomo\'sci Matematyczne 48 (2), 157-171
		Polskie Towarzystwo Matematyczne,
		2012.

		\bibitem{Ma4}
		Lech Maligranda
		\textit{On submultiplicativity of an $N$-function and its conjugate},
		Aequationes mathematicae 89, 569-573,
		2015.

		\bibitem{Mar}
		Józef Marcinkiewicz
		\textit{Sur l'interpolation I},
		Studia Math. 6 (1), 1-17
		Institute of Mathematics Polish Academy of Sciences,
		1936.

		\bibitem{MaZy}
		Józef Marcinkiewicz, Antoni Zygmund
		\textit{Mean values of trigonometrical polynomials},
		Fundamenta Mathematicae 28, 131-166
		Institute of Mathematics Polish Academy of Sciences,
		1937.

		\bibitem{RaRe}
		Malempati M. Rao, Zhong D. Ren
		\textit{Theory of Orlicz Spaces},
		Chapman \& Hall Pure and Applied Mathematics (Book 146)
		1991.

		\bibitem{Ste} 
		Elias M. Stein, 
		\textit{Singular Integrals and Differentiability Properties of Functions},
		Princeton University Press, Princeton, New Jersey
		1970.

		\bibitem{Zyg} 
		Antoni Zygmund, 
		\textit{Trigonometric Series},
		Third Edition, Volumes I \& II combined,
		Cambridge University Press, 
		1959.

\end{thebibliography}
\end{document}